\documentclass[a4paper,11pt]{amsart}
\usepackage{cancel}

\DeclareFontFamily{U}{mathx}{}
\DeclareFontShape{U}{mathx}{m}{n}{<-> mathx10}{}
\DeclareSymbolFont{mathx}{U}{mathx}{m}{n}
\DeclareMathAccent{\widehat}{0}{mathx}{"70}
\DeclareMathAccent{\widecheck}{0}{mathx}{"71}
\DeclareMathAccent{\widetilde}{0}{mathx}{"72}

\usepackage[square,numbers]{natbib}

\usepackage{amsfonts,amsmath,amssymb,amsthm}
\usepackage{mathtools}
\usepackage{xargs,xifthen}
\usepackage{xcolor}
\usepackage{graphics}
\usepackage{subcaption}
\usepackage{enumitem}
\usepackage[hidelinks]{hyperref}
\usepackage[noabbrev]{cleveref}
\usepackage{tikz}
\usepackage{pgfplots}
\pgfplotsset{compat=1.18}
\usepgfplotslibrary[groupplots]
\usepackage{tikzscale}

\usepackage{comment}

\newtheorem{theorem}{Theorem}

\newtheorem{corollary}[theorem]{Corollary}

\theoremstyle{definition}

\newtheorem{example}[theorem]{Example}
\theoremstyle{remark}
\newtheorem{remark}[theorem]{Remark}

\newtheorem{assumption}[theorem]{Assumption}
\numberwithin{theorem}{section}
\numberwithin{equation}{section}
\numberwithin{table}{section}
\numberwithin{figure}{section}

\newcommand{\wrt}{with respect to}

\newcommand{\naturalnumbers}{\mathbb{N}}

\newcommand{\Domain}{\Omega}

\DeclareMathOperator{\identity}{id}

\DeclareMathOperator{\divergence}{div}

\DeclareMathOperator{\curl}{curl}
\DeclareMathOperator{\kernel}{ker}
\DeclareMathOperator{\spann}{span}
\DeclareMathOperator*{\argmin}{argmin}

\DeclareMathOperator{\supp}{supp}
\DeclareMathOperator{\interior}{int}

\newcommand{\diffd}{\,\mathrm{d}}
\newcommand{\dx}{\diffd x}

\newcommandx{\partialt}[1][1=]{\partial_t^{#1}}

\newcommand{\MeshFont}[1]{\mathcal{#1}}
\newcommand{\coarse}[1]{{#1}_{H}}

\newcommand{\Mesh}{\MeshFont{T}}

\newcommand{\Element}{K}
\newcommand{\varElement}{G}

\newcommand{\finalTime}{T}
\newcommandx{\Nb}[2][1=,2=\Element]{\mathtt{N}^{#1}(#2)}

\newcommandx{\timetheta}[1][1=n]{t_{#1}^{\theta}}

\newcommandx{\ContSpace}[1][1=\Domain]{C(#1)}
\newcommandx{\SmoothSpace}[2][2=\Domain]{C^{#1}(#2)}
\newcommandx{\Ltwo}[1][1=\Domain]{L^2(#1)}
\newcommandx{\Lebesgue}[2][2=\Domain]{L^{#1}(#2)}
\newcommandx{\Hone}[1][1=\Domain]{H^1(#1)}
\newcommandx{\Hloc}[1][1=\Domain]{H^1_0(#1)}
\newcommandx{\Hdiv}[1][1=\Domain]{H(\divergence;#1)}
\newcommandx{\Hcurl}[1][1=\Domain]{H(\curl;#1)}
\newcommandx{\Sobolev}[3][3=\Domain]{W^{#1,#2}(#3)}
\newcommandx{\Sobolevtwo}[2][2=\Domain]{H^{#1}(#2)}

\DeclarePairedDelimiterX{\norm}[1]{\|}{\|}{#1}
\DeclarePairedDelimiterX{\seminorm}[1]{|}{|}{#1}
\DeclarePairedDelimiterX{\abs}[1]{|}{|}{#1}
\NewDocumentCommand{\Norm}{O{} m m}{\norm[#1]{#2}_{#3}}
\NewDocumentCommand{\SemiNorm}{O{} m m}{\seminorm[#1]{#2}_{#3}}
\NewDocumentCommand{\LtwoNorm}{O{} m O{\Domain}}{\norm[#1]{#2}_{\Ltwo[{#3}]}}
\NewDocumentCommand{\HlocNorm}{O{} m O{\Domain}}{\norm[#1]{\nabla #2}_{\Ltwo[{#3}]}}

\DeclarePairedDelimiterX{\bilinear}[2]{(}{)}{#1,#2}
\DeclarePairedDelimiterX{\dualpair}[2]{\langle}{\rangle}{#1,#2}
\NewDocumentCommand{\IP}{O{} m m m}{\bilinear[#1]{#2}{#3}_{#4}}
\NewDocumentCommand{\LtwoIP}{O{} m m O{\Domain}}{\IP[#1]{#2}{#3}{\Ltwo[#4]}}
\NewDocumentCommand{\HlocIP}{O{} m m}{\IP[#1]{\nabla #2}{\nabla #3}{\Ltwo}}
\NewDocumentCommand{\energyIP}{O{} m m}{a\bilinear[#1]{#2}{#3}}
\NewDocumentCommand{\restrictedenergyIP}{O{} m m m}{a\vert_{#2}\bilinear[#1]{#3}{#4}}
\NewDocumentCommand{\stabilizationIP}{O{} m m}{s\bilinear[#1]{#2}{#3}}
\NewDocumentCommand{\DP}{O{} m m m}{\dualpair[#1]{#2}{#3}_{#4}}

\newcommand{\BulkSolution}{u}
\newcommand{\msBulkSolution}{\tilde{u}_H}

\newcommand{\PDEcoefficient}{A}
\newcommand{\RHS}{f}

\newcommand{\ConstantInit}{C_{\mathrm{init}}}

\newcommand{\OperatorFont}[1]{\mathcal{#1}}
\newcommand{\LtwoProjection}{\Pi}

\newcommandx{\locCorrector}[1][1=\ell]{\OperatorFont{C}^{[#1]}}

\newcommandx{\locElementCorrector}[1][1=\ell]{\OperatorFont{C}_{\Element}^{[#1]}}

\newcommandx{\locOMCorrector}[1][1=\ell]{\OperatorFont{R}^{[#1]}}

\newcommandx{\locEnrichedCorrector}[1][1=\mathrm{loc}]{\OperatorFont{D}^{#1}}
\newcommandx{\ElementEnrichedCorrector}[1][1=\varElement]{\OperatorFont{D}_{#1}}
\newcommandx{\locElementEnrichedCorrector}[2][1=\varElement,2=\lambda]{\OperatorFont{D}_{#1}^{[#2]}}

\newcommandx{\locmsProjection}[1][1=\ell]{\widetilde{\OperatorFont{P}}^{[#1]}}
\newcommandx{\EnrichedCorrectionMapping}[1][1=j]{\widetilde{\OperatorFont{Q}}^{#1}}
\newcommandx{\EnrichedMSMapping}[1][1=j]{\widecheck{\OperatorFont{Q}}^{#1}}
\newcommandx{\locEnrichedCorrectionMapping}[1][1=j]{\widetilde{\OperatorFont{Q}}^{#1,\mathrm{loc}}}
\newcommandx{\locEnrichedMSMapping}[1][1=j]{\widecheck{\OperatorFont{Q}}^{#1,\mathrm{loc}}}
\newcommandx{\enrichedmsRitzProjection}[1][1=j]{\widecheck{\OperatorFont{P}}^{#1}}
\newcommandx{\locenrichedmsRitzProjection}[1][1=j]{\widecheck{\OperatorFont{P}}^{\mathrm{#1,loc}}}
\newcommandx{\enrichedRitzProjection}[1][1=j]{\widehat{\OperatorFont{P}}^{#1}}
\newcommandx{\locenrichedRitzProjection}[1][1=j]{\widehat{\OperatorFont{P}}^{\mathrm{#1,loc}}}

\newcommand{\BulkSpace}{V}

\newcommand{\kernelSpace}{W}
\newcommandx{\enrichedCorrectionSpace}[1][1=j]{\coarse{\widehat{\kernelSpace}}^{#1}}
\newcommandx{\locenrichedCorrectionSpace}[1][1=j]{\coarse{\widehat{\kernelSpace}}^{#1,\mathrm{loc}}}
\newcommandx{\enrichedmsBulkSpace}[1][1=j]{\coarse{\widecheck{\BulkSpace}}^{#1}}
\newcommandx{\locenrichedmsBulkSpace}[1][1=j]{\coarse{\widecheck{\BulkSpace}}^{#1,\mathrm{loc}}}

\newcommand{\BulkTestFunction}{v}

\newcommand{\wt}[1]{\widetilde{\mathbf{#1}}}

\usepackage{bm}

\usepackage{scalerel}
\newcommand{\invdag}{\scalerel*{\rotatebox[origin=c]{180}{$\dagger$}}{\dagger}}
\definecolor{greeen}{RGB}{0,127,0}

\begin{document}

\title[Enriched higher-order multiscale methods for waves]{Enriched higher-order multiscale approaches with applications to  wave propagation}
\author[B.~Kalyanaraman, F.~Krumbiegel, R.~Maier, S.~Wang]{Balaje~Kalyanaraman$^\ddagger$$^\mathsection$, Felix~Krumbiegel$^{\invdag}$, Roland~Maier$^\dagger$, and Siyang~Wang$^\ddagger$}
\address{${}^{\ddagger}$ Department of Mathematics and Mathematical Statistics, Ume\aa{} University,  901 87 Ume\aa, Sweden.}
\email{\{balaje.kalyanaraman,siyang.wang\}@umu.se}
\address{${}^{\mathsection}$ Department of Computing Science, Ume\aa{} University,  901 87 Ume\aa, Sweden.}
\address{${}^{\dagger}$ Institute for Applied and Numerical Mathematics, Karlsruhe Institute of Technology, Englerstr.~2, 76131 Karlsruhe, Germany.}
\email{roland.maier@kit.edu}
\address{${}^{\invdag}$ Department of Mathematics, Saarland University, Campus E 1.1, 66123 Saarbrücken, Germany.}
\email{felix.krumbiegel@uni-saarland.de}
\date{\today}
%
%
\begin{abstract}
  We consider the numerical solution of partial differential equations  with coefficients that are strongly heterogeneous in space. We provide an overview of higher-order localized orthogonal decomposition (LOD) methods for the elliptic setting, including recent advancements, and then present a generalization of the strategy to linear hyperbolic multiscale problems. We address the limitations of earlier constructions for the wave equation, which only achieve second-order convergence in space, independent of the chosen polynomial degree. Building on the methodology of enriched corrections recently developed for parabolic multiscale problems, we motivate and propose an enriched higher-order LOD method for the wave equation. The enriched corrections exhibit exponential decay and can be computed on patches. Under minimal assumptions on the coefficient and standard well-preparedness conditions on the data, we derive a priori error estimates that achieve optimal high-order convergence rates, thereby overcoming the previously observed saturation of the convergence rate. With the fifth-order Rosenbrock-Wanner (ROW) time integrator, we conduct a series of numerical examples to verify our theoretical results. We provide examples showing the optimal spatial convergence of the method including the localization errors for different polynomial orders. We also present examples showing the optimal convergence rates of the time discretization.
\end{abstract}

\maketitle




\section{Introduction}

Wave propagation in strongly heterogeneous materials arises in many applications, such as composite materials, geophysical subsurface modeling, and medical imaging. In these settings, the underlying media exhibit fine-scale heterogeneities that are modeled by highly varying coefficients in the wave equation.  Standard numerical techniques require the discretization scale to be fine enough to resolve these fine-scale properties, leading to prohibitively high computational cost. 

Many multiscale methods, addressing the issue of highly varying coefficients in the wave equation, have been developed in the last years, see, e.g.,~\cite{Jiang2010,Abdulle2011,Engquist2011,Owhadi2017,Abdulle2017,MaiP19,MaierVerfurth2022,GeeM23,Kru26}. Among these approaches, the localized orthogonal decomposition (LOD) method has proved to be effective for solving partial differential equations with rough $L^\infty$-coefficients. The method was originally developed for elliptic problems \cite{MalP14,HenP13}, see also the monograph~\cite{MalP20} and the review article~\cite{Altmann2021}. In its original formulation, the LOD method is based on standard  first-order continuous finite elements on a coarse mesh, which are adjusted with appropriate localized fine-scale corrections.
Based on the ideas of the LOD and gamblets~\cite{Owh17,OwhS19}, a higher-order LOD method for elliptic problems was proposed in~\cite{Mai21} and later improved in~\cite{DonHM23,HauLM25}. Its main idea is to extract convergence rates from regularity of the right-hand side only, while requiring minimal regularity assumptions on the coefficient and the solution. This is possible by an appropriate construction of a coarse-scale space, which has favorable orthogonalization properties and can be computed by solving a set of localized sub-problems. However, as the construction is tailored to elliptic problems, it is not reliable to obtain a higher-order convergence behavior if used for the spatial discretization in the context of the wave equation, see~\cite{KruM25}. More precisely, higher-order spatial regularity of the solution is required to achieve higher-order convergence rates, which cannot be expected for problems with rough $L^\infty$-coefficients. In the very general setting, the convergence rate in space does not exceed order two. 

Recently, we have introduced an enriched higher-order LOD method for parabolic problems. It demonstrates that higher-order convergence rates can be obtained by enriching the original higher-order LOD space with additional, problem-adapted corrections~\cite{KalKMW26}. The enriched corrections exploit higher temporal regularity of the solution -- which is reasonable under appropriate assumptions on the initial conditions -- instead of higher-order spatial regularity. In addition, the enriched corrections are proved to have an exponential decay property that allows the calculation of the enriched corrections on the same patch as the higher-order LOD basis, thus preserving the computational structure and scalability of the LOD framework. 

The goal of this work is twofold. First, we aim at providing a review of different constructions of higher-order LOD methods. Second, we aim at generalizing the enrichment methodology to the wave equation to overcome the previously observed saturation of the convergence rate in higher-order LOD discretizations. Under minimal assumptions on the coefficient and standard well-preparedness of the data, we prove optimal higher-order convergence rates in the energy norm with respect to the spatial discretization. 
For higher order accuracy in time, we employ a suitable higher-order time stepping method. Several choices exist in the literature, for example, the backward differentiation formulas~(BDF)~\cite{shampine1997matlab, gear1971simultaneous}, the fully-implicit Runge-Kutta methods~(FIRK)~\cite{hairer1999stiff, AdaptiveRadauPaper}, or the Rosenbrock-Wanner~(ROW) methods~\cite{rosenbrock1963some,steinebach2023construction}. Other non-stiff choices are also popular for second-order differential equations, like the Runge-Kutta-Nystr\"om methods~\cite{fehlberg1969klassische, MONTIJANO2024115533}. In this work, we use the Rosenbrock-Wanner methods, in particular, the 5th-order accurate L-stable \texttt{Rodas5P} scheme.

The rest of the paper is organized as follows. In Section~\ref{sec:2}, we present the model problem and provide an overview of higher-order LOD methods. In particular, we present the overall methodology and focus on different possible localization approaches. In Section~\ref{sec:3}, we present an enriched higher-order LOD method for the wave equation, which enriches the multiscale space from the elliptic setting with appropriate additional localized corrections. We prove optimal convergence rates in the energy norm for an ideal (non-localized) construction and motivate its localization. Numerical examples are presented in Section~\ref{sec:4} to verify the theoretical results. Finally, we draw conclusions in Section~\ref{sec:5}.

\subsubsection*{Notation.} We write~$a\lesssim b$ if~$a\leq Cb$, where~$C$ is a generic constant that is independent of the mesh sizes~$H$ and $h$, the localization parameter~$\ell$, and the fine-scale parameter~$\varepsilon$. The constant may depend on the domain~$\Omega$, the regularity of the mesh~$\mathcal{T}_H$, the polynomial degree~$p$, and the dimension~$d$. If it depends on the final time~$T$, we write $a \lesssim_T b$. Further, we use~$a\sim b$ if~$a\lesssim b \lesssim a$. 

\section{Higher-order localized orthogonal decomposition methods}\label{sec:2}

In this section, we present a higher-order multiscale method that constructs problem-adapted approximation spaces based on the general ideas of the LOD~\cite{MalP20} and gamblets~\cite{OwhS19}. We first introduce an \emph{ideal} approach in the elliptic setting and review different strategies to make it practically feasible. We then discuss the applicability of the strategy to the acoustic wave equation. As it turns out, the higher-order rates are not necessarily obtained in such a setting. Adjustments to retain the desired rates will be introduced in the subsequent section.

\subsection{Prototypical construction}
The localized orthogonal decomposition method was first introduced in the context of elliptic model problems by M\aa{}lqvist and Peterseim~\cite{MalP14} and later improved by Henning and Peterseim in~\cite{HenP13}. In a prototypical setting, the method decomposes the solution space into a coarse space with favorable approximation properties and a remainder space that is the orthogonal complement of the coarse space \wrt{} the elliptic bilinear form
\begin{equation}
  a(u,v) = (A\nabla u, \nabla v)_{\Ltwo}
\end{equation}
for~$u,v\in\Hloc$. Here, $\Omega \subset \mathbb{R}^d$ is an open and bounded Lipschitz domain in some dimension $d$ and $A\in L^\infty(\Omega)$ with~$0<\alpha\leq A(x)\leq \beta<\infty$ for almost every~$x\in\Domain$. Note that matrix-valued coefficients could be considered as well, but we stick to scalar-valued choices to simplify the presentation. %
Implicitly, we have in mind coefficients~$A$ that are heterogeneous on one or multiple scales and denote with~$\varepsilon$ the finest such scale. However, the existence of an explicit scale~$\varepsilon$ is not necessary for the construction. 
The general goal is to construct the coarse space such that the $a$-orthogonal projection of any $v \in \Hloc$ preserves certain coarse-scale properties of $v$ (such as, e.g., its $L^2$-projection onto a coarse space). This is beneficial to obtain favorable approximation properties of the coarse space. 
While the classical LOD~\cite{MalP14,HenP13,MalP20} considers a first-order continuous polynomial space as the foundation, the approach has been extended to higher-order discontinuous polynomial spaces in~\cite{Mai21,DonHM23} and to higher-order continuous finite element spaces in~\cite{HauLM25}. The original higher-order construction in~\cite{Mai21} provides the basis for all these constructions and is presented in the following. 

We start from a discontinuous higher-order polynomial space on a quasi-uniform regular quadrilateral mesh~$\coarse{\Mesh}$, i.e., 
\begin{equation}
  \coarse{V} \coloneqq \{v\in\Ltwo\mid v\vert_K\in\mathbb{Q}_p(K),K\in\coarse{\Mesh}\},
\end{equation}
where $\mathbb{Q}_p$ refers to the polynomial space with partial degree up to $p \in \mathbb{N}_0$. 
We define the~$L^2$-projection~$\coarse{\Pi}\colon\Ltwo\to\coarse{V}$ onto the polynomial space by 
\begin{equation}
  (\coarse{\Pi}v,\coarse{v})_{\Ltwo} = (v,\coarse{v})_{\Ltwo}
\end{equation}
for all~$\coarse{v}\in\coarse{V}$. Since the space consists of discontinuous polynomials, the projection is a local operation in the sense that it can be decomposed into element-wise contributions. On each element, we have the following properties for the~$L^2$-projection: for any element~$K\in\coarse{\Mesh}$, we have~$L^2$-stability, an approximation estimate, and an inverse inequality, i.e., 
\begin{subequations}\label{eq:Ltwo_projection_properties}
  \begin{alignat}{2}
    \| \coarse\Pi v\|_{L^2(K)} &\leq \| v \|_{L^2(K)},\qquad &&v\in\Ltwo,\label{eq:Ltwo_projection_stability}\\
    \| v - \coarse{\Pi}v\|_{L^2(K)} &\lesssim H^k \| v \|_{H^k(K)},\qquad &&v\in H^k(\coarse{\Mesh}),\quad k \leq p+1,\label{eq:Ltwo_projection_approximation}\\
    \| \nabla (\coarse\Pi v)\|_{L^2(K)} &\lesssim H^{-1}\| \coarse\Pi v \|_{L^2(K)},\qquad &&v\in\Ltwo,\label{eq:Ltwo_projection_inverse}
  \end{alignat}
\end{subequations}
where the piecewise Sobolev spaces are defined for~$k\geq 1$ by 
\begin{equation}
  H^{k}(\coarse{\Mesh}) \coloneqq \{v\in\Ltwo\mid v\vert_K\in H^k(K),K\in\coarse{\Mesh}\},\qquad H^0(\Domain)\coloneqq\Ltwo.
\end{equation}
The higher-order LOD method as introduced in~\cite{Mai21} now constructs a finite-dimensional multiscale space by minimizing the elliptic energy under a constraint. That is, we define an operator~$\mathcal{R}\colon\Ltwo\to\Hloc$ such that for any~$v\in\Ltwo$ we have 
\begin{equation}\label{eq:minimization}
  \mathcal{R}v \coloneqq \argmin_{w\in\Hloc} a(w, w)\quad\textrm{s.t.}\quad\coarse{\Pi}w = \coarse{\Pi}v.
\end{equation}
From a more practical point of view, the minimization problem~\eqref{eq:minimization} can be rephrased as a saddle point problem. That is, we seek $(\mathcal{R}v,\xi) \in \Hloc \times V_H$ such that
\begin{equation}\label{eq:saddle_pointR}
  \begin{alignedat}{3}
    &a(\mathcal{R}v, w) && + (\xi,w)_{L^2(\Omega)} &&= 0,\\
    &(\mathcal{R}v, \mu)_{L^2(\Omega)} && && = (v, \mu)_{L^2(\Omega)}
  \end{alignedat}
\end{equation}
for all $w \in \Hloc,\,\mu \in V_H$. 
The multiscale space is then defined by~$\coarse{\widetilde{V}} = \mathcal{R}\coarse{V}$. A basis of the space is given by $\{\mathcal{R}\Lambda_i\}_i$ if $\{\Lambda_i\}_i$ is a basis of $V_H$.
The space $\coarse{\widetilde{V}}$ has many favorable properties. First, it has the same dimension as the original polynomial space~$\coarse{V}$ and inherits its coarse-scale properties due to the constraint imposed on the functions. Moreover, by design the space~$\coarse{\widetilde{V}}$ is conforming. Finally, we have by~\eqref{eq:saddle_pointR} that 
\begin{equation}\label{eq:ortho} 
    \coarse{\widetilde{V}} \bot_a W,\qquad \text{where}\quad W \coloneqq \ker \Pi_H\vert_{H^1_0(\Omega)}. 
\end{equation}
From this, we directly obtain that $\mathcal{R}v \in \coarse{\widetilde{V}}$ is the orthogonal projection of $v$ onto $\coarse{\widetilde{V}}$.  

The above-mentioned properties lead to the following consequences in the context of elliptic problems: We consider for some~$f\in\Ltwo$ the solution~$u\in\Hloc$ to
\begin{equation}
  a(u,v) = (f, v)_{\Ltwo}
\end{equation}
for all~$v\in\Hloc$. Provided that $f$ is piecewise sufficiently regular, we obtain with $u - \mathcal R u \in W$, the orthogonality~\eqref{eq:ortho}, and~\eqref{eq:Ltwo_projection_properties} 
\begin{equation}\label{eq:error_elliptic}
  \begin{aligned}
    \|\nabla(u-\mathcal{R}u)\|_{\Ltwo}^2 &\lesssim a(u-\mathcal{R}u,u-\mathcal{R}u) = a(u, u- \mathcal{R}u) = (f, u-\mathcal{R}u)_{\Ltwo}\\
    & = (f- \coarse{\Pi}f, (u-\mathcal{R}u) - \coarse{\Pi}(u-\mathcal{R}u))_{\Ltwo}
    \\&\lesssim H^{k} \|f\|_{H^k(\coarse{\Mesh})} H \|\nabla (u-\mathcal{R}u)\|_{\Ltwo},
  \end{aligned}
\end{equation}
where~$k \leq p+1$. That is,
\begin{equation}\label{eq:errideal}
    \|\nabla(u-\mathcal{R}u)\|_{\Ltwo} \lesssim H^{k+1}\|f\|_{H^k(\coarse{\Mesh})}.
\end{equation}
Therefore, the function $\mathcal{R}u$ is a higher-order approximation of~$u$.  
Note that by design $\mathcal{R}u$ is also the Galerkin approximation in~$\coarse{\widetilde{V}}$. 
We emphasize that~\eqref{eq:errideal} holds independently of the regularity of the solution~$u$. That is, $u \in \Hloc$ is sufficient and regularity of the coefficient~$A$ beyond~$L^\infty(\Domain)$ is not required. Moreover, the construction is not limited by strong heterogeneities of the coefficient. 
This is essential in the context of multiscale problems, where for strongly heterogeneous materials classical finite element methods would require the mesh size to globally resolve all fine oscillations even if only coarse approximations are of interest. %
The construction is therefore particularly useful for time-dependent problems where it is favorable to set up smaller linear systems as they need to be solved in every time step. 

\subsection{Localization}\label{subsec:localization}
To construct the multiscale space~$\coarse{\widetilde{V}}$, we have to solve the constrained minimization problem~\eqref{eq:minimization} (or rather the saddle-point problem~\eqref{eq:saddle_pointR}) for every  basis function in $\coarse{V}$. Practically, this is done by discretizing $\Hloc$ on a mesh whose mesh size is small enough to resolve all heterogeneities of the coefficient~$A$. 
This results in a large computational overhead compared to a classical finite element method on a fine mesh. More importantly, this is practically infeasible since we aim at avoiding (global) fine-scale computations completely.\footnote{We do not further discuss the fine-scale discretization for the sake of readability. We emphasize, however, that the effect of the fine-scale discretization can be rigorously taken into account as well, see, e.g., \cite{Mai21}. In essence, all the results can be transferred to the fully discretized setting.}

Fortunately, we can leverage two further advantages of the method. First, the minimization problem can be computed trivially parallel for all basis functions in $\coarse{V}$. 
Second, the solutions to the minimization problem~\eqref{eq:minimization} exhibit an exponential decay property. That is, away from the support of a local basis function $\Lambda$, the contribution of~$\mathcal{R}\Lambda$ is exponentially small as stated in Theorem~\ref{thm:decay} below. %
Such a property is crucial for the effectiveness of the method. The minimization problem needs to be computed only on small subdomains, which greatly reduces the computational overhead. %
The local computational region is given for~$\ell\in\naturalnumbers$ by the~\emph{$\ell$-Patch}~$\Nb[\ell][S]$ around a subdomain~$S\subset\Omega$, defined by the recursion 
\begin{equation}\label{eq:patch}
  \Nb[1][S] \coloneqq \interior \bigcup \{\overline{\Element}\in\coarse{\Mesh} \mid \overline{\Element}\cap \overline{S}\neq\emptyset \},\qquad \Nb[\ell+1][S] = \Nb[1][{\Nb[\ell][S]}],\quad \ell\in\naturalnumbers.
\end{equation}
In the following, we abbreviate~$\Nb[][S]=\Nb[1][S]$, and~$\Nb[0][S] = S$. 
Then, we have the following decay result from~\cite[Thm.~4.1]{Mai21}. 
\begin{theorem}[exponential decay]\label{thm:decay}
  Let~$\ell\in\naturalnumbers$, $K\in\coarse\Mesh$, and~$\Lambda\in\coarse V$ be a basis function with $\supp \Lambda = K$. Then, there exists a generic constant~$C>0$ such that 
  \begin{equation}
    \|\nabla\mathcal{R}\Lambda\|_{L^2(\Domain\setminus\Nb[\ell])} \lesssim \exp(-C\ell) \|\nabla\mathcal{R}\Lambda\|_{\Ltwo}.
  \end{equation}
\end{theorem}%
The decay result in Theorem~\ref{thm:decay} motivates the localization of the minimization problem~\eqref{eq:minimization} as described in the following. We define the operator~$\widetilde{ \mathcal{R}}^{[\ell]}_K\colon L^2(K)\to H^1_0(\Nb[\ell])$ such that for any~$v\in\Ltwo$ we have 
\begin{equation}\label{eq:locRK}
  \widetilde {\mathcal{R}}^{[\ell]}_Kv \coloneqq \argmin_{w\in H^1_0(\Nb[\ell])} a(w, w)\quad\textrm{s.t.}\quad\coarse{\Pi}w = \coarse{\Pi}(v\vert_K).
\end{equation}
Equivalently, the problem can be rephrased as a localized saddle point problem similar as in~\eqref{eq:saddle_pointR}. 
The corresponding localized multiscale operator~$\widetilde {\mathcal{R}}^{[\ell]}$ is then given as the sum of the element-wise localized contributions, i.e., 
\begin{equation}\label{eq:deftRell} 
  \widetilde{ \mathcal{R}}^{[\ell]} = \sum_{K\in\coarse{\Mesh}} \widetilde{\mathcal{R}}^{[\ell]}_K,
\end{equation} 
and the localized multiscale space is defined by $\widetilde{\mathcal{R}}^{[\ell]} \coarse{V}$. While the optimal convergence of the ideal approach in~\eqref{eq:errideal} does not hold anymore, it holds up to an exponential error in~$\ell$. In particular, for $\ell\sim C_p|\log H|$, the higher-order rates can be recovered. Here, the constant $C_p$ scales at most quadratically with $p$, see~\cite[Cor.~4.6]{Mai21}. Although a certain scaling with $p$ is clearly necessary due to the envisioned higher-order rates, numerical experiments indicate that the theoretically predicted quadratic scaling seems to be suboptimal. 

We emphasize that the simple localization strategy defined in~\eqref{eq:locRK} introduces  a polluting effect in the localization error, see~\cite[Thm.~4.4]{Mai21}. That is, the total error in $H^1(\Omega)$ scales as $\mathcal{O}(H^{p+2} + H^{-1}\exp(-C\ell))$, where we omit the precise scaling of $p$ in the constants. In particular, for a fixed $\ell$ the errors may increase when decreasing the mesh size $H$. This is not a theoretical artifact, as such a behavior is also observed numerically, see~\cite[Fig.~5.2]{Mai21} and~\cite[Fig~7.2]{DonHM23}. 

\subsection{Stabilized localization}
To avoid the problematic scaling of $\mathcal O(H^{-1})$ in the localization error, specifically designed stabilized variants of the localization strategy can be used. In the following, we present the two stabilized strategies developed in~\cite{DonHM23} and \cite{HauLM25}, respectively. For the former, we only provide a high-level introduction, and present the latter (more general) approach in more detail. 

\subsubsection{Extended bubble functions}\label{subsubsec:extended_bubble}
First, we consider the stabilization used in~\cite{DonHM23}, which is based on ideas presented in~\cite{Altmann2021,HauP22}. 
To understand the following construction, we first connect the ideal construction in~\eqref{eq:saddle_pointR} to the classical LOD,  which is defined in terms of \emph{corrections}, see~\cite{MalP20}. That is, we would like to write $\mathcal R = \identity - \mathcal C$ for an appropriate operator $\mathcal C$. However, such a construction is only feasible for $\Hloc$-conforming functions. Therefore, we replace $\coarse{V} \not\subset \Hloc$ by a space $\coarse{U} \subset \Hloc$ with the property that there is a one-to-one relation between a polynomial basis function $\Lambda \in \coarse{V}$ and a basis function $\Phi \in \coarse{U}$, characterized by $\Pi_H\Phi = \Lambda$ and $\supp \Phi = \supp \Lambda$. The existence of such a \emph{bubble space} follows from~\cite[Cor.~3.6]{Mai21}. By construction, we have that $\mathcal R \coarse{V} = \mathcal R \coarse{U}$ and the restriction $\mathcal R\vert_{\coarse{U}}\colon \coarse{U} \to \Hloc$ allows for the splitting $\mathcal R  \vert_{\coarse{U}} = \identity - \mathcal C$, where $\mathcal C\colon \coarse{U} \to W = \ker \Pi_H\vert_{\Hloc}$ is defined, for any $\coarse{b} \in \coarse{U}$ by
\begin{equation}\label{eq:corr}
    a(\mathcal C \coarse{b},w) = a(\coarse{b},w)
\end{equation}
for all $w \in W$. This resembles the classical setting of the LOD. The corresponding localization of element-wise contributions recovers the space~$\widetilde{ \mathcal{R}}^{[\ell]} \coarse{V}$ with~$\widetilde{ \mathcal{R}}^{[\ell]}$ defined in~\eqref{eq:deftRell}. Recall, however, that this construction is subject to the suboptimal behavior of the localization error with a pre-factor~$H^{-1}$. The factor can be avoided by a simple trick: we relax the condition $\supp \Phi_{K,0} = \supp \Lambda_{K,0}$ on all the \emph{bubble functions} $\Phi_{K,0}$ associated to the zero-order basis functions $\Lambda_{K,0}$, $K \in \coarse{\Mesh}$. More precisely, we allow a larger support in~$\Nb[1]$, see Figure~\ref{fig:xbubble} for an illustration in one dimension. In particular, this slightly modifies the localization strategy, which is applied element-wise as for the classical LOD. For a precise construction, we refer to~\cite{DonHM23}. Therein, it is shown that this small modification cures the suboptimal scaling both theoretically and practically. 

Note that all but the zero-order basis functions of the localized multiscale space can be computed as defined in~\eqref{eq:locRK}. Only the zero-order functions require the explicit construction of (extended) bubble functions.  

\begin{figure}[!t]
    \centering
    \includegraphics[width=\linewidth]{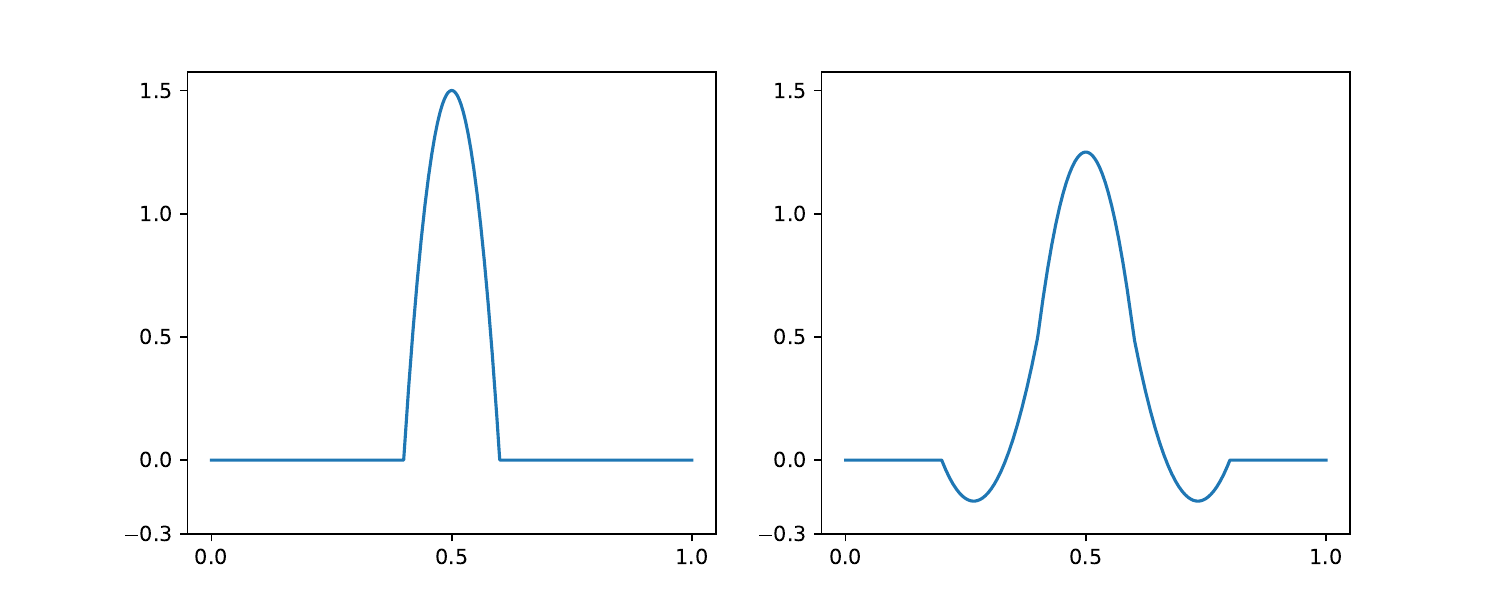}%
    \caption{Illustration of a zero-order bubble function (left) and its stabilized (extended) version (right) for~$p=0$, and~$H=\tfrac15$ in one dimension. 
    }
    \label{fig:xbubble}
\end{figure} 

\subsubsection{A generalized framework}
To avoid the explicit construction of bubble functions, another localization strategy is presented in~\cite{HauLM25} based on a clever splitting of the operator $\mathcal R$ introduced in~\eqref{eq:minimization}. In the following, this strategy is briefly summarized. 
Let~$\coarse{\mathcal{I}}\colon \Ltwo \to \coarse{V}^\mathrm{cg}$ be a local quasi-interpolation operator that maps into piecewise (multi-)linear polynomials that are additionally conforming. A possible choice is $\coarse{\mathcal{I}} = \coarse{\mathcal{E}}\circ \coarse{\Pi}^{0}$, where~$\coarse{\Pi}^{0}$ is the~$L^2$-projection onto piecewise constant polynomials and~$\coarse{\mathcal{E}}$ an averaging operator that maps into piecewise (multi-)linear polynomials that are overall continuous and fulfill zero boundary conditions. More precisely, for any~$v\in\Ltwo$ at an interior node~$z\in\coarse\Mesh$, we have with the definition of patches in~\eqref{eq:patch} that 
\begin{equation}
  (\coarse{\mathcal{I}}v)(z) \coloneqq \sum_{K\in\Nb[][\{z\}]}\frac{|K|}{|\Nb[][\{z\}]|} \int\limits_K v \dx,
\end{equation}
see~\cite{ErnG17} for further details on such constructions. 
The main idea is now to do a similar splitting of $\mathcal R$ as in Section~\ref{subsubsec:extended_bubble}, but replacing the identity operator. For the non-localized method, we set $\mathcal R = \coarse{\mathcal I} - \mathcal K $, where $\mathcal K \colon \Ltwo \to \Hloc$ is defined, for any $v \in \Ltwo$, via the pair $(\mathcal K v, \lambda) \in \Hloc \times \coarse{V}$ that solves
\begin{equation}\label{eq:saddle_pointglobal}
  \begin{alignedat}{3}
    &a(\mathcal{K}v, w) && + (w,\lambda)_{\Ltwo} &&= a (\coarse{\mathcal{I}}v, w),\\
    &(\mathcal{K}v, \mu)_{\Ltwo} && && = -c(v,\mu)
  \end{alignedat}
\end{equation}
for all $w \in \Hloc$ and $\mu \in \coarse{V}$. Here, $c(v,\mu) \coloneqq (v-\coarse{\mathcal{I}}v, \mu)_{\Ltwo}$. Note that $\mathcal R = \coarse{\mathcal I} - \mathcal K $ holds by design. Importantly, this splitting involves two operators that map into $\Hloc$ without constructing bubble functions. 
To construct a localized operator, we define
\begin{equation*} 
    \mathcal{K}^{[\ell]}_K\colon\Ltwo\to H^1_0(\Nb[\ell])
\end{equation*}
for any~$v\in\Ltwo$ via the solution pair 
\begin{equation*} 
    (\mathcal{K}^{[\ell]}_Kv,\lambda_K^{[\ell]}) \in H^1_0(\Nb[\ell]) \times \coarse{V}(\Nb[\ell])
\end{equation*}
that solves
\begin{equation}\label{eq:saddle_point}
  \begin{alignedat}{3}
    &a(\mathcal{K}^{[\ell]}_Kv, w) && + (w,\lambda_K^{[\ell]})_{L^2(\Nb[\ell])} &&= a_K (\coarse{\mathcal{I}}v, w),\\
    &(\mathcal{K}^{[\ell]}_Kv, \mu)_{L^2(\Nb[\ell])} && && = -c_K(v,\mu)
  \end{alignedat}
\end{equation}
for all~$w\in H^1_0(\Nb[\ell])$ and~$\mu\in\coarse{V}(\Nb[\ell])$, 
where the restricted bilinear forms are defined by 
\begin{equation}
  a_K (v,w) \coloneqq (A\nabla v, \nabla w)_{L^2(K)},\qquad c_K(v,\mu) = (v-\coarse{\mathcal{I}}v, \mu)_{L^2(K)}.
\end{equation}
The localized operator~$\mathcal{R}^{[\ell]}$ is then given by~$\mathcal{R}^{[\ell]} \coloneqq \coarse{\mathcal{I}} - \mathcal{K}^{[\ell]}$, where 
\begin{equation*} 
    \mathcal{K}^{[\ell]} \coloneqq \sum_{K\in\coarse{\Mesh}}\mathcal{K}^{[\ell]}_K.
\end{equation*}
%
%
The construction of a higher-order localized multiscale space is completed by finally re-defining~$\coarse{\widetilde{V}}\coloneqq\mathcal{R}^{[\ell]}\coarse{V}$. %
This space can now be used in a Galerkin fashion for an elliptic model problem and yields higher-order convergence rates, only relying on the regularity of the right-hand side, provided that $\ell$ is chosen appropriately. 

\subsection{Reduced convergence for the wave equation}
We now discuss the use of the previously introduced multiscale space~$\coarse{\widetilde{V}}$ in the context of the acoustic wave equation with a strongly heterogeneous coefficient. Generally, time-dependent problems are ideal to employ multiscale spaces since the overhead of constructing problem-adapted spaces quickly pays off by smaller linear systems in every time step. This is especially appealing if higher-order convergence rates in space can be obtained. 
However, the above construction suffers from a reduced convergence rate when applied to the acoustic wave equations, see~\cite{KruM25} for further details. Theoretically and practically, it turns out that the multiscale space is not well-suited for the wave equation in the very general setting, where we only have~$A\in L^\infty(\Domain)$. Note that this is independent of the chosen time discretization. %

In the following, we investigate the use of the higher-order multiscale space for the discretization in space and illustrate the above-mentioned reduced convergence rates. We consider the weak formulation of the wave equation that seeks a function~$u\colon [0,T]\to\Hloc$ such that for almost every~$t\in[0,T]$ 
\begin{equation}\label{eq:variational_wave}
  ( \ddot u(t),v)_{H^{-1}(\Domain)\times\Hloc} + a(u(t),v) = (f(t),v)_{\Ltwo}
\end{equation}
for all~$v\in\Hloc$ with $\Domain$ as before. Provided that~$f\in L^2(0,T;\Ltwo)$ and~$u$ fulfills the initial conditions~$u(0)=u_0\in\Hloc$ and~$\dot u(0)=v_0\in\Ltwo$, the solution is unique~(\cite[Ch.~7.2]{Eva10}) with 
\begin{equation}
  u\in L^2(0,T;\Hloc),\quad \dot u\in L^2(0,T;\Ltwo),\quad \ddot u\in L^2(0,T;H^{-1}(\Domain)).
\end{equation}
To have sufficient regularity in time, we require well-preparedness and compatibility conditions as in~\cite{Abdulle2017}. These conditions essentially control the norms of the solution~$u$ and its time-derivatives at the initial time by forcing the solution to be `in line' with the heterogeneity of the coefficient. As an example, smooth initial conditions are not to be expected in a strongly heterogeneous material, see also~\Cref{rem:compatibility}. %

The required assumptions are summarized in the following and will be employed for the remainder of this work. %
\begin{assumption}[Compatibility and well-preparedness]\label{ass:regularity}
  Let $k,m\in\naturalnumbers_0$ such that
  \begin{enumerate}[label={(A\arabic*)},ref=\thedefinition~(A\arabic*)]
    \crefalias{enumi}{assumption}
    \item $\RHS\in\SmoothSpace{m}[{[0,T];\Sobolevtwo{k}[\coarse{\Mesh}]}]$,
    \item $\BulkSolution(0) = \BulkSolution_0\in\Hloc$, $\partialt\BulkSolution(0) = \BulkTestFunction_0\in\Hloc$,
    \item \label{ass:compatibility} $\partialt[\nu]\BulkSolution(0) = \partialt[{\nu-2}] \RHS(0) - \divergence (\PDEcoefficient \nabla \partialt[{\nu-2}]\BulkSolution(0))\in\Hloc$ for $\nu=2, \dots, m$,
    \item $\partialt[m+1]\BulkSolution(0) = \partialt[{m-1}] \RHS(0) - \divergence (\PDEcoefficient \nabla \partialt[{m-1}]\BulkSolution(0))\in\Ltwo$. 
    \item Further, suppose that there exists an $\varepsilon$-independent constant $\ConstantInit$ such that
    \begin{equation*}
      \sum_{\nu=0}^{m}\Norm{\partialt[\nu] \BulkSolution(0)}{\Sobolevtwo{1}} + \Norm{\partialt[m+1]\BulkSolution(0)}{\Ltwo} \leq \ConstantInit.
    \end{equation*}
  \end{enumerate}
\end{assumption}
\begin{remark}\label{rem:compatibility}
  We note that in \Cref{ass:regularity}, (A1), the regularity~$k\in\naturalnumbers_0$ only contributes to the (improved) spatial convergence behavior and is not necessary for the existence or uniqueness of a solution. 
  The other assumptions are to be understood as compatibility and  well-preparedness conditions. That is, initial conditions need to suitably match the heterogeneity of the underlying coefficient, which is physically reasonable. We emphasize that the assumptions can easily be satisfied with zero initial conditions, in which case a wave is only generated by the external source term over time. 
  
  The compatibility conditions also ensure that the temporal derivatives of a solution solve a wave equation as well, which can be used to derive the regularity results
  \begin{equation*}
    \partialt[m]u\in L^2(0,T;\Hloc),\quad \partialt[m-1] u\in L^2(0,T;\Ltwo),\quad \partialt[m-2] u\in L^2(0,T;H^{-1}(\Domain)).
  \end{equation*}
\end{remark}
To discretize~\eqref{eq:variational_wave}, we first define a semi-discrete higher-order LOD method. That is, we seek a function~$\coarse{\tilde{u}}\colon [0,T]\to\coarse{\widetilde{V}}$ such that for almost every~$t\in[0,T]$ 
\begin{equation}\label{eq:lod_variational_wave}
  (\partialt[2] \coarse{{\tilde{u}}}(t),\coarse{\tilde{v}})_{\Ltwo} + a(\coarse{\tilde{u}}(t),\coarse{\tilde{v}}) = (f(t),\coarse{\tilde{v}})_{\Ltwo}
\end{equation}
for all~$\coarse{\tilde{v}}\in\coarse{\widetilde{V}}$, where the initial conditions are given by~$\coarse{\tilde{u}}(0)=\mathcal{R}^{[\ell]} u_0$ and~$\partialt \coarse{\tilde{u}}(0)=\mathcal{R}^{[\ell]} v_0$. 

As mentioned above, under minimal regularity assumptions on the coefficient only reduced convergence rates are obtained for the semi-discrete method defined in~\eqref{eq:lod_variational_wave}. This is also confirmed by numerical investigations, see~\cite{KruM25} and \Cref{sec:4} below. We illustrate the reason in the following theorem, for which we consider the ideal non-localized setting ($\ell = \infty$) for better readability.

\begin{theorem}[second-order convergence]\label{thm:waves:sd_error_plod}
  Let~$A \in L^\infty(\Domain)$, $\ell = \infty$, and suppose that~\Cref{ass:regularity} holds for some~$k\in\naturalnumbers$, $k \leq p + 1$, and~$m\geq 4$. Further, let~$\BulkSolution$ be the solution to~\eqref{eq:variational_wave} and~$\msBulkSolution$ the solution to~\eqref{eq:lod_variational_wave}. Then for~$t\in[0,T]$, we have 
  \begin{equation}\label{eq:waves:sd_error_plod}
    \LtwoNorm{\partialt(\BulkSolution(t) - \msBulkSolution(t))} + \HlocNorm{(\BulkSolution(t) - \msBulkSolution(t))} \lesssim_\finalTime H^2.
  \end{equation}
\end{theorem}

\begin{proof}[Sketch of the proof]
  Since the higher-order multiscale space $\coarse{\widetilde{V}}$ is a conforming space with respect to $\Hloc$, we can use standard energy techniques to obtain  
  \begin{equation}
    \LtwoNorm{\partialt(\BulkSolution(t) - \msBulkSolution(t))} + \HlocNorm{(\BulkSolution(t) - \msBulkSolution(t))} \lesssim_T \|\nabla (u(t) - \mathcal{R}u(t))\|_{\Ltwo}.
  \end{equation}
  For the error on the right-hand side, we use the fact that~$\mathcal{R}$ is the orthogonal projection onto the multiscale space. This gives (similarly as in~\eqref{eq:error_elliptic})
  \begin{equation}\label{eq:noHO}
  \begin{aligned}
     \|\nabla (u(t) - \mathcal{R}u(t))\|_{\Ltwo}^2 &\lesssim a(u(t)-\mathcal{R}u(t),u(t)-\mathcal{R}u(t)) = a(u(t), u(t)- \mathcal{R}u(t))\\
      &=(f(t), u(t)-\mathcal{R}u(t)_{\Ltwo} - (\partialt[2] u(t), u(t)-\mathcal{R}u(t))_{\Ltwo}\\
      &\lesssim H^{k + 1} \|f(t)\|_{H^{k}(\coarse{\Mesh})} \|\nabla (u(t)-\mathcal{R}u(t))\|_{\Ltwo} \\&\qquad\qquad+ H^2\|\nabla \partialt[2] u(t)\|_{\Ltwo}\|\nabla (u(t)-\mathcal{R}u(t))\|_{\Ltwo},
    \end{aligned}
  \end{equation} 
  where the last estimate follows from the fact that~$u(t)-\mathcal{R}u(t)\in\ker \coarse\Pi$ and the approximation property~\eqref{eq:Ltwo_projection_approximation}. Here, we observe that the reduced convergence comes from the fact that, for very general coefficients, one cannot expect more regularity than~$\partialt[2] u(t)\in\Hloc$. Only under additional regularity assumptions on the coefficient, higher orders can be extracted from spatial regularity of~$\partialt[2] u$. We refer to~\cite{KruM25} for further details. 
\end{proof}

\section{Enriched higher-order spaces}\label{sec:3}
The inability to obtain higher-order convergence rates in the time-dependent setting originates from the construction of the multiscale space, which is specifically designed  for the elliptic setting. From a more practical point of view, this means that the space does not capture well all the spatial effects that build up throughout time. In the analysis, the effect shows in~\eqref{eq:noHO}, where higher orders cannot be extracted from spatial regularity of~$\partial_t^2 u$. Therefore, we aim at adjusting the multiscale space to suitably avoid this problematic term. The idea is to \emph{enrich} the localized space $\coarse{\widetilde V}$ with well-chosen additional functions, which has been introduced in~\cite{KalKMW26} in the context of parabolic problems. 
This section provides an overview of the methodology and shows that it can be adjusted to the wave equation as well. 

We start by defining the enriched multiscale space based on~\emph{additional correction operators}, which are recursively applied. These corrections are defined locally and can be interpreted as implicitly constructing an asymptotic expansion of the solution. We further motivate the construction in \Cref{ss:motivation} below and outline the main steps of the error analysis in \Cref{thm:waves:sd_error_eholod}. %

First, we define the additional element-wise correction
\begin{equation*} 
    \mathcal{D}^{[\ell]}_K\colon L^2(\Nb[\ell])\to H^1_0(\Nb[\ell])\cap\kernel\coarse{\Pi}
\end{equation*}
for any~$v\in L^2(\Nb[\ell])$ as the solution $\mathcal{D}^{[\ell]}_Kv \in H^1_0(\Nb[\ell])\cap\kernel\coarse{\Pi}$ to
\begin{equation}\label{eq:additional_correction}
  a(\mathcal{D}^{[\ell]}_Kv, w) = -(v, w)_{L^2(\Nb[\ell])}
\end{equation}
for all~$w\in H^1_0(\Nb[\ell])\cap\kernel\coarse{\Pi}$. As we only need to apply the additional correction operator to basis functions of the localized multiscale space~$\coarse{\widetilde{V}}$, a summation over multiple patches is not required. Recall that 
\begin{equation}
  \coarse{\widetilde{V}} = \spann \bigcup_{K\in\coarse{\Mesh}}\{\mathcal{R}^{[\ell]}_K\Lambda_{K,i}\}_{i=1}^M, 
\end{equation}
where~$M\coloneqq(p+1)^d$ and $\Lambda_{K,i}$ are standard (local) basis functions of~$\coarse{V}$. 
For~$\nu\in\naturalnumbers$, we define
\begin{equation}\label{eq:enriched_correction_space}
  \coarse{\widetilde{W}}^\nu \coloneqq \spann \bigcup_{K\in\coarse{\Mesh}}\{(\mathcal{D}^{[\ell]}_K)^\nu\mathcal{R}^{[\ell]}_K\Lambda_{K,i}\}_{i=1}^M \eqqcolon \spann \bigcup_{K\in\coarse{\Mesh}}\{\widetilde{\Lambda}_{K,i}^\nu\}_{i=1}^M, 
\end{equation}
and the \emph{enriched multiscale space} for~$j\in\naturalnumbers$ by 
\begin{equation}\label{eq:enriched_multiscale_space}
  \coarse{\widetilde{V}}^j \coloneqq \coarse{\widetilde{V}} + \coarse{\widetilde{W}}^1 + \dots + \coarse{\widetilde{W}}^j.
\end{equation}
For $j = 0$, we formally set $\coarse{\widetilde{V}}^0 \coloneqq \coarse{\widetilde{V}}$. 
To summarize the construction, the first step is to solve the saddle point problem~\eqref{eq:saddle_point} for each  basis function~$\Lambda_{K,i}\in\coarse{V}$. This gives the multiscale basis functions~$\widetilde{\Lambda}_{K,i}$, which are defined on the patches~$\Nb[\ell]$. In the elliptic setting, this space yields higher-order convergence rates, but the space is not sufficient for higher-order rates in the context of time-dependent PDEs. Thus, we correct the basis functions~$\widetilde{\Lambda}_{K,i}$ by solving~\eqref{eq:additional_correction} to obtain new functions $\widetilde{\Lambda}_{K,i}^1$. Note that the patches~$\Nb[\ell]$ are not increased in the process. This procedure may be repeated multiple times. All constructed functions are then combined to form a new enriched multiscale space. %

It turns out that increasing the dimension of the discrete space via enrichments is favorable compared to a decrease of the mesh size. Moreover, increasing~$p$ is also beneficial compared to a decrease of the mesh size, see \cite[Fig.~4.2]{KalKMW26} for an illustration.  

\subsection{Improved error estimates}\label{ss:motivation}
So far, it is not obvious why the construction of the enriched multiscale space defined in~\eqref{eq:enriched_multiscale_space} should solve the issue of stagnating orders of convergence. To provide some overall intuition, we therefore consider a fixed~$t\in[0,T]$ and set~$\ell=\infty$, which corresponds to global patches~$\Nb[\infty]=\Domain$ in~\eqref{eq:additional_correction}. The global operator $\mathcal D \colon \Ltwo \to \Hloc$ is independent of $K \in \coarse{\Mesh}$. For $v \in \Ltwo$, we obtain $\mathcal D v \in \Hloc$ by solving 
\begin{equation}\label{eq:additional_correctionGlobal} 
  a(\mathcal{D} v, w) = -(v, w)_{\Ltwo}
\end{equation}
for all $w \in \Hloc \cap \ker \coarse{\Pi}$. With~$\mathcal R$ defined in~\eqref{eq:minimization}, we define an expansion of the solution~$u$ to~\eqref{eq:variational_wave} of order $j$ by 
\begin{equation}\label{eq:defQ}
  \mathcal{Q}^ju(t) \coloneqq \sum_{\nu=0}^j \mathcal{D}^\nu\mathcal{R}(\partialt[2\nu]u(t)),
\end{equation}
which is well-defined if~$u$ is sufficiently smooth in time, see also~\Cref{ass:regularity}. Note that $\mathcal{Q}^ju(t) \in \coarse{\widetilde{V}}^j$. 
Therefore, the space $\coarse{\widetilde{V}}^j$ is a reasonable approximation space provided that $u \approx \mathcal{Q}^ju$ for an appropriate choice of $j$. 
For~$k=p+1$, we have for~$j=0$ 
\begin{equation}
  \|\nabla (u(t)-\mathcal{Q}^0 u(t))\|_{\Ltwo} \lesssim H^{p+2} \|f(t)\|_{H^{p+1}(\coarse{\Mesh})} + H^2\|\nabla \partialt[2]u(t)\|_{\Ltwo},
\end{equation}
see the proof of~\Cref{thm:waves:sd_error_plod}. 
For~$j=1$, one can show that the additional term in~\eqref{eq:defQ} leads to the improved error estimate 
\begin{equation*}
  \begin{aligned}
    \|\nabla (u(t)&-\mathcal{Q}^1 u(t))\|_{\Ltwo} \\&\lesssim H^{p+2} \|f(t)\|_{H^{p+1}(\coarse{\Mesh})} + H^2\|\nabla (\partialt[2]u(t) - \mathcal{R}(\partialt[2]u(t)))\|_{\Ltwo}\\
    &\lesssim H^{p+2} \|f(t)\|_{H^{p+1}(\coarse{\Mesh})} + H^{p+4} \|\partialt[2] f(t)\|_{H^{p+1}(\coarse{\Mesh})} + H^4\|\nabla \partialt[4]u(t)\|_{\Ltwo}.
  \end{aligned}
\end{equation*}
By applying the procedure recursively, we may obtain  
\begin{equation}
  \|\nabla (u(t)-\mathcal{Q}^j u(t))\|_{\Ltwo} \lesssim H^{p+2} + H^{2j+2}.
\end{equation}
More details on these improved error estimates are presented in \Cref{thm:waves:sd_error_eholod} below.

Before we state the precise error estimate, we define the semi-discrete \emph{enriched higher-order LOD method} for the wave equation: we seek a function~$\coarse{\tilde{u}}\colon (0,T)\to\coarse{\widetilde{V}}^j$ such that for~$t\in[0,T]$ we have %
\begin{equation}\label{eq:eholod_variational_wave}
  (\partialt[2] \coarse{{\tilde{u}}}(t),\coarse{\tilde{v}})_{\Ltwo} + a(\coarse{\tilde{u}}(t),\coarse{\tilde{v}}) = (f(t),\coarse{\tilde{v}})_{\Ltwo}
\end{equation}
for all~$\coarse{\tilde{v}}\in\coarse{\widetilde{V}}^j$, where the initial conditions are given by~$\coarse{\tilde{u}}(0)=\mathcal{Q}^j u_0$ and~$\partialt \coarse{\tilde{u}}(0)=\mathcal{Q}^j v_0$. As mentioned above, in the following theorem we only treat the non-localized setting with $\ell = \infty$ for illustration purposes. 

\begin{theorem}[higher-order convergence]\label{thm:waves:sd_error_eholod}
  Let~\Cref{ass:regularity} hold for some~$k\in\naturalnumbers$ and~$m\geq 2j + 4$ and set $\ell = \infty$. Further, let~$\BulkSolution$ and~$\msBulkSolution$ be the solutions to~\eqref{eq:variational_wave} and~\eqref{eq:eholod_variational_wave}, respectively. If~$j\geq\lceil\frac{k-1}{2}\rceil$, then 
  \begin{equation}\label{eq:waves:sd_error_eholod}
    \LtwoNorm{\partialt(\BulkSolution(t) - \msBulkSolution(t))} + \HlocNorm{(\BulkSolution(t) - \msBulkSolution(t))} \lesssim_\finalTime H^{k+1}
  \end{equation}
  for almost all~$t\in[0,T]$ and~$k \leq p+1$. 
\end{theorem}

\begin{proof}[Sketch of the proof]
  The proof uses similar arguments as in~\cite{KalKMW26} and is adapted here to the wave equation, see also~\cite{Kru26}. Note that by standard arguments therein, the error $\BulkSolution(t) - \msBulkSolution(t)$ can be bounded by the error~$u(t) - \mathcal{Q}^ju(t)$ of the expansion~\eqref{eq:defQ} up to a generic constant. We now show that the expansion~$\mathcal{Q}^ju$ actually converges with an improved rate. 

  Let~$t\in[0,T]$ be arbitrary but fixed. By definition of the operators~$\mathcal{R}$ and~$\mathcal{D}$, we have 
  \begin{equation}\label{eq:waves:eholod_mapping_error_kernel_space}
    \psi(t) \coloneqq \BulkSolution(t) - \mathcal{Q}^ju(t)\in \Hloc \cap \kernel \coarse{\Pi}. 
  \end{equation}
  In the following, we omit the argument~$t$ for ease of presentation. From direct calculations, we obtain 
  \begin{equation}\label{eq:waves:enriched_ms_mapping_estimate}
    \begin{aligned}
      \|\nabla\psi\|_{\Ltwo}^2 \lesssim \energyIP{\psi}{\psi} = \LtwoIP{\RHS}{\psi} - \LtwoIP{\partialt[2]\BulkSolution}{\psi} - \energyIP{\mathcal{Q}^j\BulkSolution}{\psi}. 
    \end{aligned}
  \end{equation}
  Since~$\psi\in\ker\coarse\Pi$, the first term on the right-hand side can be optimally bounded by
  \begin{equation}\label{eq:waves:prot_enriched_mapping_rhs_estimate}
    \begin{aligned}
      \LtwoIP{\RHS}{\psi} &= \LtwoIP{(\identity - \coarse\LtwoProjection)\RHS}{(\identity - \coarse\LtwoProjection)\psi} \lesssim H^{k+1} \Norm{\RHS}{\Sobolevtwo{k}[\coarse\Mesh]} \HlocNorm{\psi}.
    \end{aligned}
  \end{equation}
  In the next step, we use the construction of~$\mathcal{D}$ and the definition of the mapping~$\mathcal{Q}^j$ in two ways. First, we  introduce the~$(j+1)$-th enriched correction as an auxiliary term in the analysis. In practice, it does not need to be computed, but sufficient temporal regularity of the solution is required. The norm of the $(j+1)$-th correction is small, as each enrichment iteration introduces a multiplicative scaling factor of order~$\mathcal{O}(H^2)$. This observation is very useful in the analysis that follows. %
  Second, we employ the fact that~$\mathcal{D}$ is constructed as an appropriate operator to -- loosely speaking -- trade time derivatives for spatial derivatives. %
  To make this concrete, we focus on the third term on the right-hand side of~\eqref{eq:waves:enriched_ms_mapping_estimate}. By the definitions of~$\mathcal{Q}^j$ and $\mathcal D$ and since~$\mathcal{R}$ maps into functions that are~$a$-orthogonal to~$\psi$, we have 
  \begin{multline}\label{eq:waves:prot_enriched_mapping_split}
  \begin{aligned}
      \energyIP{\mathcal{Q}^ju}{\psi} &= \energyIP{\sum_{\nu=0}^j \mathcal{D}^\nu\mathcal{R}(\partialt[2\nu]u)}{\psi} \\&= \energyIP{\sum_{\nu=1}^{j+1} \mathcal{D}^\nu\mathcal{R}(\partialt[2\nu]u)}{\psi} - \energyIP{\mathcal{D}^{j+1}\mathcal{R}(\partialt[2(j+1)]u)}{\psi} \\&
      = - \LtwoIP{\partialt[2]\mathcal{R}u}{\psi} - \LtwoIP{\partialt[2] \sum_{\nu=2}^{j+1} \mathcal{D}^{\nu-1}\mathcal{R}(\partialt[2(\nu-1)]u)}{\psi}\\&\qquad\qquad + \LtwoIP{\partialt[2] \mathcal{D}^{j}\mathcal{R}(\partialt[2j]u)}{\psi}.
  \end{aligned}
  \end{multline}
  Plugging~\eqref{eq:waves:prot_enriched_mapping_rhs_estimate} and~\eqref{eq:waves:prot_enriched_mapping_split} back into~\eqref{eq:waves:enriched_ms_mapping_estimate}, we get
  \begin{subequations}\label{eq:mapping_recursion}
  \begin{equation}
      \begin{aligned}
          \|\nabla\psi\|_{\Ltwo}^2 &\lesssim H^{k+1} \Norm{\RHS}{\Sobolevtwo{k}[\coarse\Mesh]} \HlocNorm{\psi} - \LtwoIP{\partialt[2]\BulkSolution}{\psi} + \LtwoIP{\partialt[2]\mathcal{R}u}{\psi}\\
          &\quad + \LtwoIP{\partialt[2] \sum_{\nu=2}^{j+1} \mathcal{D}^{\nu-1}\mathcal{R}(\partialt[2(\nu-1)]u)}{\psi} - \LtwoIP{\partialt[2] \mathcal{D}^{j}\mathcal{R}(\partialt[2j]u)}{\psi},
      \end{aligned}
  \end{equation}
  and using~$\psi = \psi - \coarse\Pi\psi$ since $\psi\in\ker\coarse\Pi$, we obtain 
  \begin{equation}
      \begin{aligned}
          \|\nabla\psi\|_{\Ltwo}^2 &\lesssim H^{k+1} \Norm{\RHS}{\Sobolevtwo{k}[\coarse\Mesh]} \HlocNorm{\psi}\\
          &\quad + H^2\LtwoNorm{\nabla \partialt[2] (u - \mathcal{R}u - \sum_{\nu=2}^{j+1} \mathcal{D}^{\nu-1}\mathcal{R}(\partialt[2(\nu-1)]u))} \HlocNorm{\psi}\\
          &\quad + H^2\LtwoNorm{\nabla \partialt[2] \mathcal{D}^{j}\mathcal{R}(\partialt[2j]u)} \HlocNorm{\psi}.
      \end{aligned}
  \end{equation}
  \end{subequations}
  Now, we can apply a recursion for the terms in the last two lines. We observe that the first term in the second line is of the same form as~$\HlocNorm{\psi}$, and with the sufficient regularity assumed in the theorem, this term can be estimated analogously to~\eqref{eq:waves:enriched_ms_mapping_estimate}--\eqref{eq:mapping_recursion}. 
  Similarly, the first term in the last line can be recursively bounded by using the definition of~$\mathcal{D}$ and the fact that $\mathcal D$ maps into~$\ker\coarse\Pi$. %
  Overall,  by repeating the above arguments we eventually obtain
  \begin{equation}
      \HlocNorm{\psi} \lesssim H^{k+1} + H^{2j+2}.
  \end{equation}
  The assertion follows by the choice~$j\geq\lceil\frac{k-1}{2}\rceil$. 
\end{proof}

\subsection{Localized enrichments}\label{sec:enriched_localization}
\Cref{thm:waves:sd_error_eholod} only considers the error of the semi-discrete method for $\ell = \infty$. As discussed in \Cref{subsec:localization}, localization is essential for ensuring the computational feasibility of  multiscale methods. 
We may use the practically feasible enriched localized multiscale space defined in~\eqref{eq:enriched_multiscale_space} because of the following localization result, which is a direct consequence of~\cite[Thm.~3.3~\&~Rem.~3.4]{KalKMW26}. The proof of the result is highly technical, and we therefore  only provide some intuition in \Cref{rem:loc} below.

\begin{corollary}[localization error of enrichments]\label{thm:localization}
    Let~$\ell\in\naturalnumbers$ with~$H^2\ell^{{d+1}}\lesssim 1$ and $\nu \in \mathbb{N}_0$. For any function $v\in\Hloc$, we have 
    \begin{equation*}
        \|{\nabla ((\mathcal{D}^{{[\ell]}})^\nu - \mathcal{D}^\nu)\mathcal{R}^{[\ell]} v} \|_{L^2(\Domain)}\lesssim \exp^{-C\ell}\|{\nabla v}\|_{L^2(\Domain)}. 
    \end{equation*}
\end{corollary}

\begin{remark}\label{rem:loc}
  As shown in~\cite{Mai21}, the operator~$\mathcal{R}$ has an exponential decay property in the sense that if the operator~$\mathcal{R}$ is applied to a function~$v$ with support only on an element~$K\in\mathcal{T}_H$, the norm of the image~$\mathcal{R} v$ decays exponentially fast away from~$K$. Such results are well-known in the context of LOD-type methods, see~\cite{MalP20,Altmann2021}. Thus, restricting computations to a small patch around the element~$K$  only leads to an exponentially small error. %
  This property, however, is crucially tied to the fact that the function~$v$ is supported locally on a single element. This means that if we were to apply such a localization result to the enriched correction~$\mathcal{D}^{[\ell]}_Kv$ once, we would instead obtain a patch of size~$2\ell$ for the same localization error due to the right-hand side already having support on the patch~$\Nb[\ell]$, see~\eqref{eq:additional_correction}. Consequently, the support grows rapidly with~$\nu$, making multiple enrichment computationally infeasible. 
  The effect is also illustrated in \Cref{fig:localization}~(left), where the original patch~$\omega_0 = \Nb[\ell]$ is shown. Applying similar localization results to elements~$G_1\in\omega_0$ produces patches~$\Nb[\ell][G_1]\subset \Nb[2\ell]$, whose union forms to a patch $\omega_1$ of increased size, and so forth. %
  
  In fact, repeated enrichment can be computed without enlarging~$\Nb[\ell]$, because  the correction is applied to a function that itself exhibits  decay. In essence, this means that correcting the contribution of a functions~$\mathcal{R}^{[\ell]}_Kv$ in some element $G_1 \in \Nb[\ell]$ only requires computations in $\Nb[\lambda][G_1]\subset\Nb[\ell]$ for some $\lambda \leq \ell$ because the decay of~$\mathcal{R}^{[\ell]}_Kv$ and $\mathcal{D}^{[\lambda]}_{G_1}\mathcal{R}^{[\ell]}_Kv$ can be combined. Note that $\lambda$ can be decreased with an increasing distance (in terms of layers of elements) between $G_1$ and $K$. This argument can be repeated also to correct the contribution of~$\mathcal{D}^{[\lambda]}_{G_1}\mathcal{R}^{[\ell]}_Kv$ on some element $G_2 \in \Nb[\lambda][G_1]$ and generally repeated multiple times. This is also illustrated in \Cref{fig:localization}~(right). Since the required choice of $\lambda$ ensures that all occurring patches are sub-patches of the original supports $\Nb[\ell]$, an element-wise compositing is not necessary and the additional corrections can be computed for the whole patch directly as defined in~\eqref{eq:additional_correction}. 

  While the above reasoning may appear intuitive, the corresponding analysis is rather lengthy and technical. We therefore refer to~\cite{KalKMW26,Kru26} for further details. 
\end{remark}

  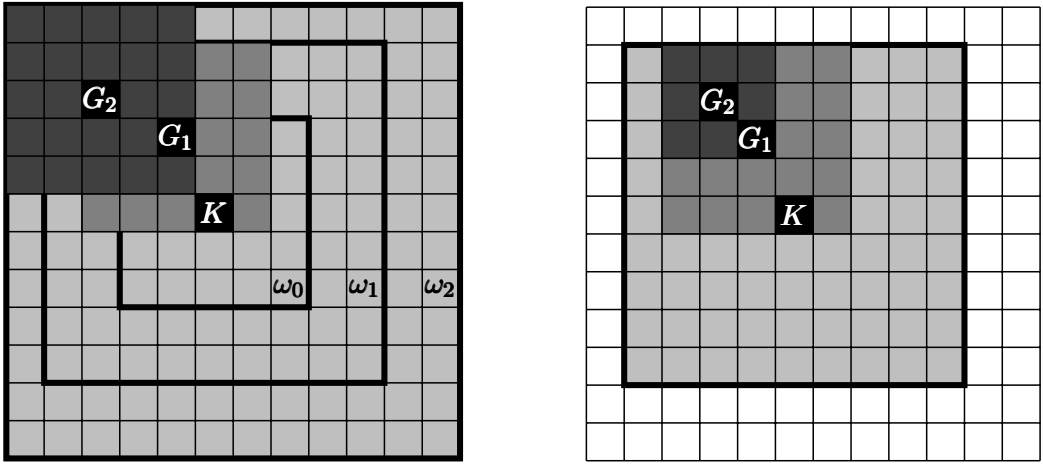
\begin{figure}
    \begin{subfigure}{.49\textwidth}
    \centering
    \begin{tikzpicture} [scale=0.5]
        \draw[line width=.8mm,fill=black!25!white] (0,0) rectangle (12,12);
        \draw[line width=.8mm,fill=black!25!white] (1,2) rectangle (10,11);
        \draw[line width=.8mm,fill=black!25!white] (3,4) rectangle (8,9);
        \draw[fill=black!50!white] (2,6) rectangle (7,11);
        \draw[fill=black!75!white] (0,7) rectangle (5,12);
        \draw[fill=black!100!white] (5,6) rectangle (6,7);
        \draw[fill=black!100!white] (4,8) rectangle (5,9);
        \draw[fill=black!100!white] (2,9) rectangle (3,10);
        \draw[line width=0.2mm, draw=black, fill=black!20!white] (0,0) grid  (12,12);
        \node[white] at (5.5, 6.5) {$\pmb{K}$};
        \node[white] at (4.5, 8.5) {$\pmb{G_1}$};
        \node[white] at (2.5, 9.5) {$\pmb{G_2}$};
        
        \node[black] at (7.5, 4.5) {$\pmb{\omega_0}$};
        \node[black] at (9.5, 4.5) {$\pmb{\omega_1}$};
        \node[black] at (11.5, 4.5) {$\pmb{\omega_2}$};
    \end{tikzpicture}%
    \end{subfigure}%
    \hfill%
    \begin{subfigure}{.49\textwidth}
        \centering
      \begin{tikzpicture} [scale=0.5]
        \draw[line width=.8mm,fill=black!25!white] (1,2) rectangle (10,11);
        \draw[fill=black!50!white] (2,6) rectangle (7,11);
        \draw[fill=black!75!white] (2,8) rectangle (5,11);
        \draw[fill=black!100!white] (5,6) rectangle (6,7);
        \draw[fill=black!100!white] (4,8) rectangle (5,9);
        \draw[fill=black!100!white] (3,9) rectangle (4,10);
        \draw[line width=0.2mm, draw=black, fill=black!20!white] (0,0) grid  (12,12);
        \node[white] at (5.5, 6.5) {$\pmb{K}$};
        \node[white] at (4.5, 8.5) {$\pmb{G_1}$};
        \node[white] at (3.5, 9.5) {$\pmb{G_2}$};
      \end{tikzpicture}
    \end{subfigure}
    \caption[Localization]{\small Illustration of two localization strategies for iterative applications of elliptic operators: a classical localization strategy (left) and the more involved localization used in~\cite{KalKMW26} that yields much smaller patches (right). } \label{fig:localization}
  \end{figure}


\section{Numerical examples}\label{sec:4}

In this section, we present some numerical examples to support our theoretical results. In time-dependent problems, the overall accuracy of the method depends on both the spatial and temporal discretization errors. In our setting, higher-order time discretizations appear natural to match the high accuracy of the enriched spatial multiscale space. Therefore, we consider the Rosenbrock-Wanner-Wolfbrandt (ROW) method applied to the semi-discrete formulation \eqref{eq:eholod_variational_wave}, 
which can be written as a system of second-order ordinary differential equations
\begin{equation}
    \mathbf{M}\,\frac{\mathrm{d}^2\wt{u}}{\mathrm{d}t^2} = -\mathbf{K}\,\wt{u} + \mathbf{f},\label{eq:ODE1}
\end{equation}
where $\mathbf{M}$ and $\mathbf{K}$ are the mass and stiffness matrix for the space $\coarse{\widetilde{V}^j}$, respectively, 
and~$\tilde{\mathbf{u}}$ the factors for the basis expansions of the function $\tilde u_H$ in~\eqref{eq:eholod_variational_wave}. As usual, $\mathbf{f}$ includes the integrals of $f$ with all the basis functions. 
In order to apply the ROW scheme, we write~\eqref{eq:ODE1} as a system of coupled first-order equations and obtain
\begin{equation}
   \begin{bmatrix}
        \mathbf{M} & \mathbf{0} \\\\
        \mathbf{0} & \mathbf{I}
   \end{bmatrix}\begin{bmatrix}
        \dfrac{\mathrm{d}\wt{v}}{\mathrm{d}t}\\\\
        \dfrac{\mathrm{d}\wt{u}}{\mathrm{d}t}
    \end{bmatrix} = 
    \begin{bmatrix}
        \mathbf{0} & -\mathbf{K} \\\\
        \mathbf{I} & \mathbf{0}
    \end{bmatrix}\begin{bmatrix}
        \wt{v}\\\\
        \wt{u}
    \end{bmatrix} + 
    \begin{bmatrix}
        \mathbf{0}\\\\
        \mathbf{f}
    \end{bmatrix},\label{eq:ODE2}
\end{equation}
which can be written more compactly as
\begin{equation}
   \wt{M}\, \frac{\mathrm{d}\wt{w}}{\mathrm{d}t} = \wt{K} \, \wt{w} + \wt{f}, \label{eq:ODE3}
\end{equation}
where
\begin{equation*} 
\wt{w} = \begin{bmatrix}
        \wt{v}\\\\
        \wt{u}
    \end{bmatrix}, \quad \wt{M} = \begin{bmatrix}
    \mathbf{M} & \mathbf{0} \\\\
    \mathbf{0} & \mathbf{I}
\end{bmatrix}, \quad
\wt{K} = \begin{bmatrix}
    \mathbf{0} & -\mathbf{K} \\\\
    \mathbf{I} & \mathbf{0}
\end{bmatrix},\quad \wt{f} = \begin{bmatrix}
        \mathbf{0}\\\\
        \mathbf{f}
    \end{bmatrix}.
\end{equation*}
An $s$-stage ROW scheme for solving~\eqref{eq:ODE3} is defined in~\citep{steinebach2023construction} and given by
\begin{align*}
    \left(\wt{M} - \tau\gamma \wt{K}\,\right)k_i &= \tau f\left(t_n + \alpha_i \tau, \,\wt{w}^n + \sum_{j=1}^{i-1} \alpha_{ij}k_j\right) + \tau\, \wt{K} \sum_{j=1}^{i-1} \gamma_{ij}k_j + \tau^2 \gamma_i\, \partialt{\wt{f}(t_n)},\\
    i &= 1, \cdots, s,\nonumber\\
    \wt{w}^{n+1} &= \wt{w}^{n} + \sum_{i=1}^s b_i k_i,
\end{align*}
where $\tau$ is the step-size and $\wt{w}^n \approx \wt{w}(t_n)$. The coefficients of the method are $\gamma, \alpha_{ij}, \gamma_{ij}$ and the $b_i$ define the weights. Moreover, we have $\alpha_i = \sum_{j=1}^{i-1} \alpha_{ij}$ and $\gamma_i = \gamma + \sum_{j=1}^{i-1} \gamma_{ij}$ with $\alpha_{ij} + \gamma_{ij} = 0$ for $i<j$. 

In this work, we use \texttt{Rodas5P}, which is an implicit 5th-order stiffly-stable ROW method, see~\cite{steinebach2023construction} for the precise construction. Note that the scheme is L-stable. Our numerical examples are based on the Julia programming language and the code is available at \href{https://github.com/Balaje/MsFEM.jl}{\texttt{https://github.com/Balaje/MsFEM.jl}}. The finite element computations are based on the \texttt{Gridap.jl} package (see~\cite{Verdugo2022,Badia2020}) and the temporal discretization makes use of the above-mentioned \texttt{Rodas5P} scheme implemented in the package \texttt{OrdinaryDiffEq.jl} (see~\cite{rackauckas2017differentialequations}). To solve the linear systems in every time step, the default LU factorization in Julia is used. All the errors are measured in the energy norm $\|\cdot\|_a \coloneqq \sqrt{a(\cdot,\cdot)}$ at the final time $T=1$.  

\begin{example}\label{example:1}
    We consider a two-dimensional example on the unit square with a piecewise constant coefficient with random values between 0.1 and 1 on a mesh of scale $\varepsilon = 2^{-6}$. Further, we use the right-hand side 
    \begin{equation*} 
    f(x,t) = \sin(\pi x_1)\sin(\pi x_2)\sin^7(t),
    \end{equation*}
    and zero initial conditions. The multiscale basis functions and the corrections are computed locally in a first-order finite element space based on a uniform Cartesian mesh with mesh size $h=2^{-7}$. A global version of the fine space is used to compute a reference solution, combined with an explicit fourth-order Runge-Kutta-Nystr\"{o}m scheme with~$\tau = 2^{-9}$ (see~\cite{MONTIJANO2024115533}). On the coarse-scale, we use the \texttt{Rodas5P} scheme with $\tau = 2^{-9}$ for the spatial convergence tests. For this example, we use the stabilization from~\cite{HauLM25}. We compare the error of the enriched higher-order multiscale method with optimally chosen correction level $j= \lceil p/2 \rceil$ (see \Cref{thm:waves:sd_error_eholod}) for selected choices of the localization parameter~$\ell$. The results are shown in \Cref{fig:4.1} for $p=1,3$, where we also present for comparison the results for smaller values of $j$ and $\ell = \infty$. In both cases, we see that the version without enrichments ($j=0$) performs worse than the versions with enrichments. Since increasing $j$ leads to an improvement by two order of $H$ directly (which is optimal for even values of $p$), it appears that the choice of $j$ can even be slightly smaller than theoretically predicted for the presented odd values of~$p$. In particular, for $p=3$ the choice $j=1$ already leads to a close-to-optimal convergence behavior, with a theoretically capped convergence rate at order 4. However, $j=0$ is clearly insufficient here to obtain optimal orders of convergence. We emphasize that this is only an effect that is observed for odd values of $p$, as even choices require the theoretically predicted choice of $j$, see, e.g.,~\cite[Fig.~3.7]{Kru26}.
\end{example}

\begin{figure}
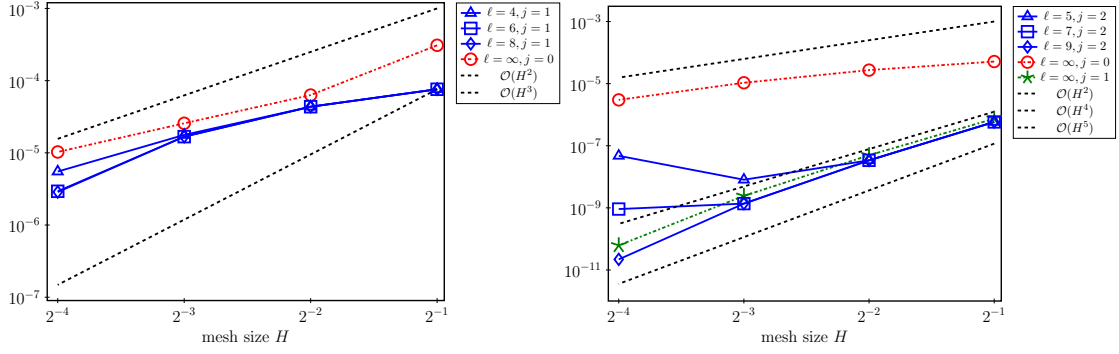

\begin{subfigure}{0.49\textwidth}
    \includegraphics[width=\textwidth]{2d-localization-p-1}
\end{subfigure}%
\begin{subfigure}{0.49\textwidth}
    \includegraphics[width=\textwidth]{2d-localization-p-3}
\end{subfigure}
\caption{Energy errors for Example~\ref{example:1} with different localization parameters for $p=1$ (left) and $p=3$ (right). The stabilization strategy proposed in \cite{HauLM25} is used in both cases.}
\label{fig:4.1}
\end{figure}

\begin{example}\label{example:2}
    In our second example, we consider the unit interval and a piecewise constant coefficient with random values between 0.1 and 1 on a mesh of size~$\varepsilon = 2^{-8}$. We choose the right-hand side
    \begin{equation*} 
    f(x,t) = \left(x + \sin(\pi x)\right)\sin^7(t),
    \end{equation*}
    and zero initial conditions. The reference scale for the computation of the multiscale basis functions, the corrections, and the reference solution is set to~$h=2^{-13}$. As before, the reference solution is computed with an explicit fourth-order Runge-Kutta-Nystr\"{o}m scheme, now with~$\tau = 2^{-13}$. On the coarse-scale, we use the \texttt{Rodas5P} scheme with $\tau = 2^{-9}$ for the spatial convergence tests. The results are shown in \Cref{fig:4.2}. In this example, we compare three different localization strategies to set up the non-enriched multiscale space: the suboptimal strategy in the original work~\cite{Mai21}, the improved version based on bubbles in~\cite{DonHM23}, and the generalized variant in~\cite{HauLM25}. We present the errors for the two polynomial degrees $p=1$ and $p=3$ with the corresponding optimal correction levels~$j$, see~\Cref{thm:waves:sd_error_eholod}.
    We observe that all strategies show the theoretically predicted convergence behavior, provided that the localization parameter~$\ell$ is chosen sufficiently large. However, the required choice of $\ell$ depends on the localization strategy: while the constructions in~\cite{DonHM23} and~\cite{HauLM25} behave very similarly, the suboptimal strategy performs significantly worse, partially by almost two orders of magnitude. For completeness, \Cref{fig:4.3} shows the minimally required choices of $\ell$ for the localization strategy of~\cite{HauLM25} for the polynomial degrees  $p=2,3,4$. 
    Note that for $p=4$ and $H = 2^{-4},2^{-5}$, we observe a stagnation of the error, 
    which is not related to the localization error. Instead, as we approach machine precision, numerical errors provide a threshold for the size of the error. Such effects can be mitigated by using quad-precision arithmetic as demonstrated in~\cite[Fig.~4.3 (right)]{KalKMW26}, although this is practically less relevant since the errors are already very small.
\end{example}

\begin{figure}
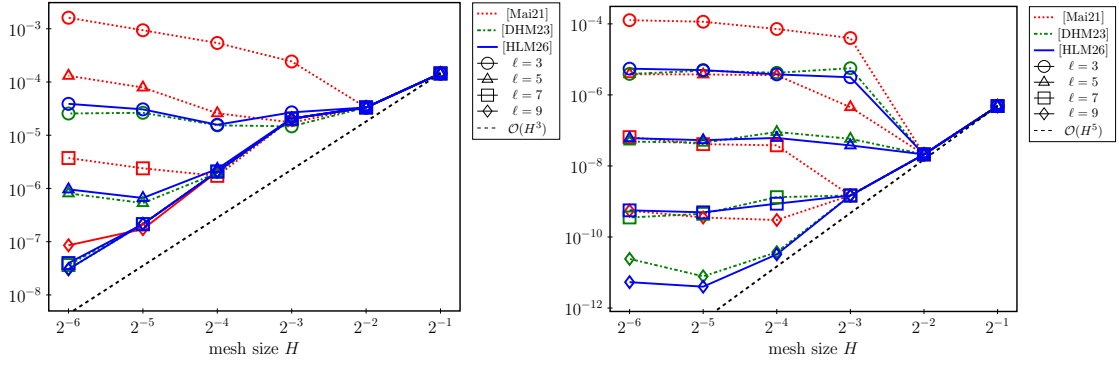

\begin{subfigure}{0.49\textwidth}
    \includegraphics[width=\textwidth]{1d-localization-p-1}
\end{subfigure}%
\begin{subfigure}{0.49\textwidth}
    \includegraphics[width=\textwidth]{1d-localization-p-3}
\end{subfigure}
\caption{Energy errors for the enriched multiscale method applied to Example~\ref{example:2} with $p=1,\, j=1$ (left) and $p=3,\, j=2$ (right), different localization parameters, and different localizations strategies.}
\label{fig:4.2}
\end{figure}

\begin{example}\label{example:3}
    Finally, we consider an example showing the optimal temporal convergence behavior of the \texttt{Rodas5P} scheme, when used together with the enriched higher-order multiscale strategy. For this example, we consider the settings from Examples~\ref{example:1}~and~\ref{example:2}, respectively. We fix the spatial discretization with the choices 
    \begin{align*}
        &h = 2^{-13},\, H=2^{-4},\, \ell = \infty,\, p = 3,\, j = 2,\quad \text{for 1D,}\\
        &h = 2^{-7},\, H=2^{-3},\, \ell = \infty,\, p = 4,\, j = 2,\quad \text{for 2D.}
    \end{align*} 
    With these choices, the spatial error due to the ehoLOD-method remains small. The results are shown in Figure~\ref{fig:4.3}~(right). In both cases, the temporal convergence rate is exactly 5, showing that the method converges optimally in time. Furthermore, the method remains stable for all chosen time-step sizes, confirming the L-stability of \texttt{Rodas5P}. 
    Note that other choices of the spatial discretization parameters work as well, but we might see a stagnation due to the spatial errors at some point.
\end{example}

\begin{figure}
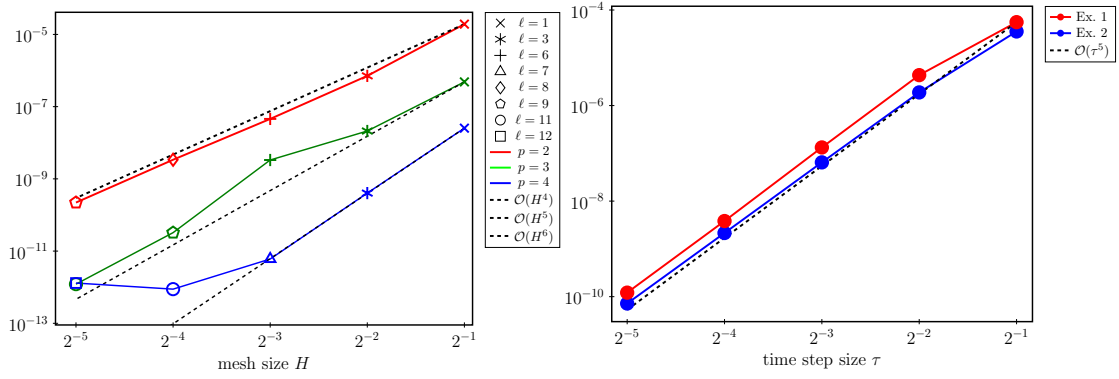

\begin{subfigure}{0.49\textwidth}
    \includegraphics[width=\textwidth]{optimal-ell-1d}
\end{subfigure}%
\begin{subfigure}{0.49\textwidth}
    \includegraphics[width=\textwidth]{temporal-convergence}
\end{subfigure}
\caption{Optimal choice of the localization parameter $\ell$ for $p=2,3,4$ and optimal $j$ in Example~\ref{example:2} (left) and the energy errors for Examples~\ref{example:1}~and~\ref{example:2} showing the temporal convergence of the \texttt{Rodas5P} scheme (right).}
\label{fig:4.3}
\end{figure}

\section{Conclusion}\label{sec:5}
In this paper, we have presented an overview of higher-order LOD methods for elliptic partial differential equations with strongly heterogeneous spatial coefficients. Moreover, using the wave equation as a model problem, we have generalized the enriched higher-order LOD framework, previously developed for parabolic problems, to the hyperbolic setting. We have established optimal convergence rates in space under minimal assumptions on the coefficient and standard compatibility and well-preparedness conditions on the data.  Finally, we have carried out numerical experiments and obtained very good agreement with the theoretical results, confirming both the accuracy and robustness of the proposed approach.

\section*{Acknowledgments}

F.~Krumbiegel and R.~Maier acknowledge funding from the Deut\-sche Forschungsgemeinschaft (DFG, German Research Foundation) -- Project-ID 258734477 -- SFB 1173. 
Parts of this work were conducted during F.~Krumbiegel's and R.~Maier's stay at the
Hausdorff Research Institute for Mathematics funded by the Deutsche Forschungsgemeinschaft (DFG, German Research Foundation) under Germany's Excellence Strategy – EXC-2047/2 – 390685813. B.~Kalyanaraman was funded by the Kempe foundation in Sweden with Project-ID JCK22-0012. Part of the computation was carried out in Project hpc2n2024-109  provided by the National Academic Infrastructure for Supercomputing in Sweden (NAISS), partially funded by the Swedish Research Council through grant agreement no.~2022-06725.


\end{document}